\documentclass[12pt]{amsart}

\usepackage{amsthm}

\usepackage{amssymb}
\usepackage{verbatim}
\usepackage[curve,matrix,arrow]{xy}
\usepackage{amsmath,amsfonts}
\usepackage{graphicx}
\usepackage{epsf}

\usepackage{graphics}

\newcommand{\mathsym}[1]{{}}

\newcommand{\thmref}[1]{Theorem~\ref{#1}}
\newcommand{\propref}[1]{Proposition~\ref{#1}}
\newcommand{\lemref}[1]{Lemma~\ref{#1}}
\newcommand{\eqnref}[1]{Equation~(\ref{#1})}
\newcommand{\remref}[1]{Remark~\ref{#1}}
\newcommand{\corref}[1]{Corollary~\ref{#1}}

\newcommand{\figref}[1]{Figure~\ref{#1}}

  {\end{list}}

\def\li{L_{i}}
\def\ri{R_{i}}

\def\oga{{\overline{\ga}}}

\def\NN{{\mathbb N}}

\newtheorem{theorem}{Theorem}[section]

\newtheorem{corollary}[theorem]{Corollary}

\newtheorem{lemma}[theorem]{Lemma}
\newtheorem{proposition}[theorem]{Proposition}

\theoremstyle{example}

\newtheorem{remark}[theorem]{Remark}

\theoremstyle{definition}

\theoremstyle{notation}

\newcommand{\hj}[3]{\hat{j}_{#1}(#2,#3)}

\newcommand{\ga}{\Gamma}

\newcommand{\tg}{\tau(\Gamma)}

\newcommand{\ee}[1]{E(#1)}
\newcommand{\vv}[1]{V(#1)}
\newcommand{\va}{\upsilon}

\newcommand{\pp}{p_{i}}

\newcommand{\qq}{q_{i}}
\newcommand{\opp}{\overline{p}_i}

\def\<{\langle }
\def\>{\rangle }
\newcommand{\secref}[1]{\S\ref{#1}}

\def\elg{\ell (\ga)}

\begin{document}

\title[Contraction Formulas For Kirchhoff And Wiener Indices ]
{Contraction Formulas For Kirchhoff And Wiener Indices}



\author{Zubeyir Cinkir}
\address{Zubeyir Cinkir\\
Zirve University\\
Faculty of Education\\
Department of Mathematics\\
Gaziantep\\
TURKEY.}
\email{zubeyirc@gmail.com}


\begin{abstract}
We establish several contraction formulas for Kirchhoff index.
We relate Kirchhoff index with some other metrized graph invariants.
By applying our contraction formulas successively when the graph is a tree, we derive new formulas for Wiener index and obtain some
previously known Wiener index formulas with new proofs.
\end{abstract}

\maketitle

\section{Introduction}\label{section introduction}

On a metrized graph $\ga$ with set of vertices $\vv{\ga}$ and resistance function $r(x,y)$, the Kirchhoff index $Kf(\ga)$ is defined as follows:
\begin{equation*}\label{eqn KIndex definition0}
\begin{split}
Kf(\ga)=\frac{1}{2}\sum_{p,\, q \in \vv{\ga}}r(p,q).
\end{split}
\end{equation*}
For the distance function $d(x,y)$ on $\ga$, the Wiener index $W(\ga)$ is defined as:
\begin{equation*}\label{eqn WIndex definition0}
\begin{split}
W(\ga)=\frac{1}{2}\sum_{p,\, q \in \vv{\ga}}d(p,q).
\end{split}
\end{equation*}
These definitions of $Kf(\ga)$ and $W(\ga)$ on a metrized graph $\ga$ agree with their usual definitions on a graph. Next, we briefly describe metrized graphs and some notations we use. Then we give a summary of the results we obtained in this paper.
%
%

A metrized graph $\ga$ is a finite connected graph equipped with a distinguished parametrization of each of its edges.
One can consider $\ga$ as a one-dimensional manifold except at finitely
many "branch points", where it looks locally like an n-pointed star.
A metrized graph $\ga$ can have multiple edges and self-loops.
For any given $p \in \ga$,
the number $\va(p)$ of directions emanating from $p$ will be called the \textit{valence} of $p$.
By definition, there can be only finitely many $p \in \ga$ with $\va(p)\not=2$.

For a metrized graph $\ga$, we will denote a vertex set for $\ga$ by $\vv{\ga}$.
We require that $\vv{\ga}$ be finite and non-empty and that $p \in \vv{\ga}$ for each $p \in \ga$ if $\va(p)\not=2$. For a given metrized graph $\ga$, it is possible to enlarge the
vertex set $\vv{\ga}$ by considering additional valence $2$ points as vertices.

For a given metrized graph $\ga$ with vertex set $\vv{\ga}$, the set of edges of $\ga$ is the set of closed line segments with end points in $\vv{\ga}$. We will denote the set of edges of $\ga$ by $\ee{\ga}$. However, if
$e_i$ is an edge, by $\ga-e_i$ we mean the graph obtained by deleting the {\em interior} of $e_i$.


We denote the length of an edge $e_i \in \ee{\ga}$ by $\li$, which represents a positive real number. The total length of $\ga$, which is denoted by $\elg$, is given by $\elg=\sum_{i=1}^e\li$.

We use the notation $\oga_i$ for the graph obtained by contracting the $i$-th edge
$e_{i}$ of a given metrized graph $\ga$ to its end
points. If $e_{i} \in \ga$ has end points $\pp$ and $\qq$, then in
$\oga_i$, these points become identical, i.e., $\pp=\qq$.
If $p$ is an end point of an edge $e_i$ in $\ga$, then by $p$ in $\vv{\oga_i}$ we mean the
vertex, in $\vv{\oga_i}$, that $p$ is contracted into.

In \secref{sec resistance function}, we briefly describe the voltage and the resistance functions on a metrized graph. We set notations concerning some specific values of these functions and recall some basic results that we use.

In \secref{sec contraction}, we improved the Kirchhoff index formulas we obtained in \cite{C6}. Then
we extended the contraction formulas obtained in \cite{C5} to bridgeless graphs. Using these results, we give a contraction formula for Kirchhoff index that involves another graph invariant $y(\ga)$ (see \eqnref{eqn definition of x and y} for the definition of $y(\ga)$ and \thmref{thm main1} for the contraction formula). This enables us giving lower and upper bounds to Kirchhoff index in terms of $y(\ga)$ and applying contraction formulas for Kirchhoff index successively (see \thmref{thm main2}).

We dealt with tree metrized graphs in \secref{sec tree}. Note that the Kirchhoff index is the same as Wiener index for a tree graph. We restate the results we derived for Kirchhoff index in \secref{sec contraction} for a tree graph. In this way, we obtain contraction formulas for the Wiener index of a tree graph. Moreover, we obtain new formulas, given in \thmref{thm wiener tree2} and \thmref{thm main3 for tree} below, for Wiener index. Our approach enables us to give new proofs of some previously known formulas,  \thmref{thm wiener tree1} and \thmref{thm wiener for tree1 multip}, for Wiener index. Then we give various examples that we apply our formulas to compute Wiener indices.
At the end of \secref{sec tree}, we state two problems. Solution to any of them will be a new proof of a conjecture about Wiener index (see \thmref{thm wiener conj} below).

\section{Resistance Function $r(x,y)$}\label{sec resistance function}

In this section, we briefly describe the resistance and the voltage functions on a metrized graph $\ga$.
We make a review of basic facts about these functions and then set the notation that we use in the rest of the paper.

For any $x$, $y$, $z$ in $\ga$, the voltage function $j_z(x,y)$ on a metrized graph
$\ga$ is a symmetric function in $x$ and $y$, which satisfies
$j_x(x,y)=0$ and $j_z(x,y) \geq 0$ for all $x$, $y$, $z$ in $\ga$.
For each vertex set $\vv{\ga}$, $j_{z}(x,y)$ is
continuous on $\ga$ as a function of all three variables.
For fixed $z$ and $y$ it
has the following physical interpretation: If $\Gamma$ is viewed
as a resistive electric circuit with terminals at $z$ and $y$,
with the resistance in each edge given by its length, then
$j_{z}(x,y)$ is the voltage difference between $x$ and $z$,
when unit current enters at $y$ and exits at $z$ (with reference
voltage $0$ at $z$).

The effective resistance between two points $x, \, y$ of a metrized graph $\ga$ is given by $r(x,y)=j_y(x,x),$
where $r(x,y)$ is the resistance function on $\ga$. The resistance function inherits certain properties of the voltage function.
For any $x$, $y$ in $\ga$,  $r(x,y)$ on
$\ga$ is a symmetric function in $x$ and $y$, and it satisfies
$r(x,x)=0$. For each vertex set $\vv{\ga}$, $r(x,y)$ is
continuous on $\ga$ as a function of two variables and
$r(x,y) \geq 0$ for all $x$, $y$ in $\ga$.
If a metrized graph $\Gamma$ is viewed as a
resistive electric circuit with terminals at $x$ and $y$, with the
resistance in each edge given by its length, then $r(x,y)$ is
the effective resistance between $x$ and $y$ when unit current enters
at $y$ and exits at $x$.

The proofs of the facts mentioned above can be found in \cite{CR} and \cite[sec 1.5 and sec 6]{BRh}.
The voltage function $j_{z}(x,y)$ and the resistance function $r(x,y)$ are also studied in the articles \cite{BF} and \cite{C1}.

We will denote by $R_i$ the resistance between the end points of an edge $e_i$ of a graph $\ga$ when the interior of the edge $e_i$ is deleted from $\ga$.

Let $\ga$ be a metrized graph with $p \in \vv{\ga}$, and let $e_i \in \ee{\ga}$ having end points $\pp$ and $\qq$.
If $\ga -e_i$ is connected, then $\ga$
can be transformed to the graph in Figure \ref{fig 2termp}
by circuit reductions. More details on this fact can be found in the articles \cite{CR} and \cite[Section 2]{C2}.
Note that in \figref{fig 2termp}, we have $R_{a_i,p} = \hj{\pp}{p}{\qq}$,
$R_{b_i,p} = \hj{\qq}{p}{\pp}$, $R_{c_i,p} = \hj{p}{\pp}{\qq}$, where $\hj{x}{y}{z}$
is the voltage function in $\ga-e_i$. We have $R_{a_i,p}+R_{b_i,p}=\ri$ for each $p \in \ga$.
\begin{remark}\label{rem notationRi}
If $\ga-e_i$ is not connected, firstly we set $R_{b_i,p}=\ri$ and $R_{a_i,p}=0$ if $p$ belongs to the component of $\ga-e_i$
containing $\pp$, and we set $R_{a_i,p}=\ri$ and $R_{b_i,p}=0$ if $p$ belongs to the component of $\ga-e_i$
containing $\qq$. Secondly, we mean $\ri \longrightarrow \infty$ in any expression that we use $\ri$.
\end{remark}
We will use these notations for the rest of the paper. Next, we recall a basic identity concerning these values:
\begin{figure}
\centering
\includegraphics[scale=1]{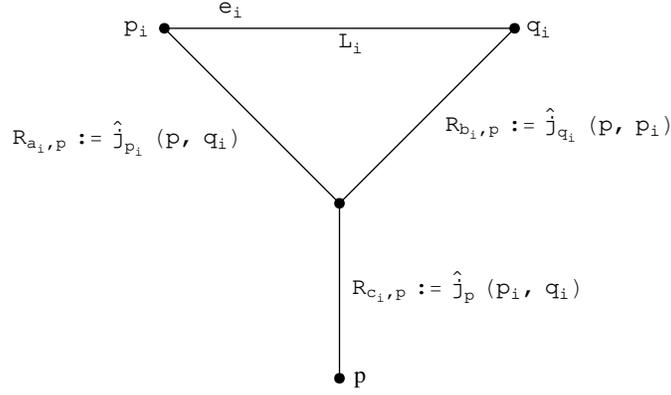} \caption{Circuit reduction of $\ga$ with reference to an edge $e_i$ and a point $p$.} \label{fig 2termp}
\end{figure}

\begin{lemma}\cite[Lemma 2.11]{C4}\label{lemrem2term}
For any $p$ and $q$ in $\vv{\ga}$,
 $$\sum_{e_i \in \, \ee{\ga}}\frac{\li(R_{a_{i},p}-R_{b_{i},p})^2}{(\li+\ri)^2}
=\sum_{e_i \in \, \ee{\ga}}\frac{\li(R_{a_{i},q}-R_{b_{i},q})^2}{(\li+\ri)^2}.$$
\end{lemma}

In the rest of the paper, for any metrized graph $\ga$ and a fixed vertex $p \in \vv{\ga}$ we will use the following notations, which we first defined in \cite{C7} and used also in \cite{C6}:
\begin{equation}\label{eqn definition of x and y}
\begin{split}
y(\ga)&=\frac{1}{4}\sum_{e_i \, \in \ee{\ga}}\frac{\li
\ri^2}{(\li+\ri)^2}+\frac{3}{4}\sum_{e_i \, \in \ee{\ga}}\frac{\li
(R_{a_i,p}-R_{b_i,p})^2}{(\li+\ri)^2},
\\x(\ga)&=\sum_{e_i \, \in
\ee{\ga}}\frac{\li^2\ri}{(\li+\ri)^2} +\frac{3}{4}\sum_{e_i \, \in
\ee{\ga}}\frac{\li \ri^2}{(\li+\ri)^2}
-\frac{3}{4}\sum_{e_i \, \in \ee{\ga}}\frac{\li
(R_{a_i,p}-R_{b_i,p})^2}{(\li+\ri)^2}.
\end{split}
\end{equation}
Note that $x(\ga)$ and $y(\ga)$ do not depend on the choice of the vertex $p$ \cite[Lemma 2.11]{C2}.
In \cite{C6}, we established connections between Kirchhoff index of $\ga$ and the invariants $x(\ga)$ and $y(\ga)$.

When we use $r_{\beta}(x,y)$, we mean the resistance function in the metrized graph $\beta$.

%
%
%

\section{Contraction Formulas For Kirchhoff Index}\label{sec contraction}

Kirchhoff index of a graph $\ga$, $Kf(\ga)$, is defined \cite{KR} as follows:
\begin{equation}\label{eqn KIndex definition}
\begin{split}
Kf(\ga):=\frac{1}{2}\sum_{p,\, q \in \vv{\ga}}r(p,q).
\end{split}
\end{equation}

The following equality was obtained in \cite[page 4038]{C6}. It gives a relation between the Kirchhoff index of $\ga$ and the Kirchhoff indexes of $\oga_i$'s. Although it is a useful formula to understand how Kirchhoff index changes after edge contractions, we can not use it for successive edge contractions because of some technical problems.
\begin{equation}\label{eqn kirchhoff contraction1}
\begin{split}
(v-2)Kf(\ga)=\sum_{e_i \in \ee{\ga}} \frac{\ri}{\li+\ri} Kf(\oga_i)+\sum_{e_i \in \ee{\ga}} \frac{\ri}{\li+\ri}\sum_{p \in \, \vv{\oga_i}} r_{\oga_i}(p,\opp).
\end{split}
\end{equation}

The idea of tracing the value of a graph invariant after successive edge contractions was successfully applied in \cite{C5}, where we studied the tau constant as an another graph invariant. We want to utilize this idea for Kirchhoff index. To do this, we first need various technical results.

The following lemma is to express $r_{\oga_i}(p,\opp)$ in terms of the resistance values on $\ga$ that we are more familiar.
\begin{lemma}\label{lem res sum}
Let $\ga$ be a metrized graph, and let $p$ be a vertex of $\ga$. For an edge $e_i$ of $\ga$ with end points $\pp$ and $\qq$, we have
$$r(\pp,p)+r(\qq,p)=2r_{\oga_i}(p,\opp)+\frac{\li \ri}{\li+\ri}-2\frac{\li R_{a_{i},p} R_{b_{i},p}}{\ri(\li+\ri)}.$$
\end{lemma}
\begin{proof}
We prove this in two cases.

\textbf{Case I:} $e_i$ is not a bridge.

From \cite[Section 2]{C4}, we have
\begin{equation}\label{eqn2term0}
\begin{split}
r(p_i,p)=\frac{(\li+R_{b_i,p})R_{a_i,p}}{\li+\ri}+R_{c_i,p}, \quad
\text{and} \quad r(q_i,p)=\frac{(\li+R_{a_i,p})R_{b_i,p}}{\li+\ri}+R_{c_i,p}.
\end{split}
\end{equation}
Thus,
\begin{equation}\label{eqn2term sum}
\begin{split}
r(\pp,p)+r(\qq,p) = \frac{\li \ri}{\li+\ri}+2 \frac{R_{a_i,p} R_{b_i,p}}{\li+\ri}+2R_{c_i,p}.
\end{split}
\end{equation}
On the other hand, from \cite[Equation 17]{C7} we have
\begin{equation}\label{eqn resistance for contraction graph}
\begin{split}
r_{\oga_i}(p,\opp) = \frac{R_{a_{i},p} R_{b_{i},p}}{\ri} + R_{c_i,p}.
\end{split}
\end{equation}
Thus, the result follows from Equations (\ref{eqn2term sum}) and (\ref{eqn resistance for contraction graph}) in this case.

\textbf{Case II:} $e_i$ is a bridge.

If $p$ belongs to the component of $\ga-e_i$ containing $\pp$, we have $r(\pp,p)+r(\qq,p)=\li+2r(\pp,p)$ and $r_{\oga_i}(p,\opp)=r(\pp,p)$.

If $p$ belongs to the component of $\ga-e_i$ containing $\qq$, we have $r(\pp,p)+r(\qq,p)=\li+2r(\qq,p)$ and $r_{\oga_i}(p,\opp)=r(\qq,p)$.

Now, we note that $\frac{\li \ri}{\li+\ri} \longrightarrow 0$ and $\frac{\li R_{a_{i},p} R_{b_{i},p}}{\ri(\li+\ri)} \longrightarrow 0$ because of \remref{rem notationRi}.

Thus, the result follows in this case, too.

\end{proof}

Now, we can substitute the value of $r_{\oga_i}(p,\opp)$ obtained from \lemref{lem res sum} into the formula given in \eqnref{eqn kirchhoff contraction1}. In this way, we derive a new formula for Kirchhoff index.
\begin{lemma}\label{lem Kirchhoff and res sum}
Let $\ga$ be a metrized graph. We have
\begin{equation*}
\begin{split}
2(v-2)Kf(\ga)&=2\sum_{e_i \in \ee{\ga}} \frac{\ri}{\li+\ri} Kf(\oga_i) + 2v \sum_{e_i \in \ee{\ga}} \frac{\li R_{a_{i},p} R_{b_{i},p}}{(\li+\ri)^2}-v \sum_{e_i \in \ee{\ga}} \frac{\li \ri^2}{(\li+\ri)^2}\\
& \qquad + \sum_{p \in \vv{\ga}}\sum_{e_i \in \ee{\ga}} \frac{\ri}{\li+\ri} (r(\pp,p)+r(\qq,p)).
\end{split}
\end{equation*}
\end{lemma}
\begin{proof}
Since $\ri = R_{a_{i},p} + R_{b_{i},p}$ for any $p \in \vv{\ga}$, we have
\begin{equation}\label{eqn expand ri}
\begin{split}
\sum_{e_i \in \ee{\ga}} \frac{\li \ri^2}{(\li+\ri)^2}=\sum_{e_i \in \ee{\ga}} \frac{\li (R_{a_{i},p} - R_{b_{i},p})^2}{(\li+\ri)^2}+4 \sum_{e_i \in \ee{\ga}} \frac{\li R_{a_{i},p} R_{b_{i},p}}{(\li+\ri)^2}.
\end{split}
\end{equation}
We note that the left hand side of \eqnref{eqn expand ri} is independent of the choice of the vertex $p$. Likewise, the first term at the right side of \eqnref{eqn expand ri} is independent of $p$ because of \lemref{lemrem2term}. Therefore,
\begin{equation}\label{eqn expand ri same}
\begin{split}
\sum_{e_i \in \ee{\ga}} \frac{\li R_{a_{i},p} R_{b_{i},p}}{(\li+\ri)^2} = \sum_{e_i \in \ee{\ga}} \frac{\li R_{a_{i},q} R_{b_{i},q}}{(\li+\ri)^2}, \qquad \text{for any vertices $p$ and $q$.}
\end{split}
\end{equation}
Now, we first multiply the equality in \lemref{lem res sum} by $\frac{\ri}{\li+\ri}$ and take the summation of the resulting equality over all edges $e_i$ in $\ee{\ga}$. Then we take the summation of the equality obtained over all vertices $p$ in $\vv{\ga}$. Finally, the result follows from \eqnref{eqn expand ri same}, \eqnref{eqn kirchhoff contraction1} and the equality we derived.
\end{proof}

Now, our goal is to simplify the formula we obtained in \lemref{lem Kirchhoff and res sum}. First, we improve a result we derived previously.

The following lemma with the condition that $\ga$ is a bridgeless metrized graph was proved in \cite[Lemma 3.10]{C7}. We note that this condition is not necessary.
\begin{lemma}\label{lem term2}
Let $\ga$ be a metrized graph, and let $\pp$ and $\qq$ be the end points of $e_i \in \ee{\ga}$. For any $p
\in \vv{\ga}$, we have
\begin{equation*}
\begin{split}
\sum_{e_i \in \,
\ee{\ga}}\frac{\li(R_{a_{i},p}-R_{b_{i},p})^2}{(\li+\ri)^2} &
=\sum_{e_i \in \, \ee{\ga}}\frac{\li}{\li + \ri}
\big(r(\pp,p)+r(\qq,p)\big) - \sum_{q \in
\vv{\ga}}(\va(q)-2)r(p,q)\\
& = 2\sum_{q \in \vv{\ga}}r(p,q) -\sum_{e_i \in \,
\ee{\ga}}\frac{\ri}{\li + \ri} \big(r(\pp,p)+r(\qq,p)\big).
\end{split}
\end{equation*}
\end{lemma}
\begin{proof}
The proof is almost the same as the proof of \cite[Lemma 3.10]{C7}. The only additional work is to use the following facts for edges that are bridges (edges whose removal disconnects the graph).

\begin{figure}
\centering
\includegraphics[scale=0.8]{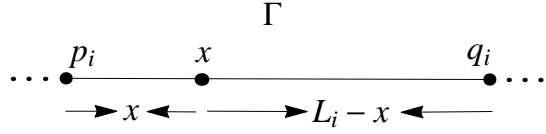} \caption{ $\ga$ with $x \in e_i$, where $e_i$ is a bridge.} \label{fig bridge1}
\end{figure}
Let $e_i \in \ee{\ga}$ be a bridge, and let $p \in \vv{\ga}$. Suppose $x \in e_i$ is as in \figref{fig bridge1} and that $e_i$ has end points $\pp$ and $\qq$.
If $p$ belongs to the component of $\ga-e_i$ containing $\pp$, we have
\begin{equation}\label{eqn bridge1}
\begin{split}
r(p,x)=r(p,\pp)+x, \qquad \frac{d}{dx}r(p,x)=1, \quad \text{and  } r(p,\pp)-r(p,\qq)=-\li.
\end{split}
\end{equation}
If $p$ belongs to the component of $\ga-e_i$ containing $\qq$, we have
\begin{equation}\label{eqn bridge2}
\begin{split}
r(p,x)=r(p,\qq)+\li-x, \qquad \frac{d}{dx}r(p,x)=-1 \quad \text{and  } r(p,\pp)-r(p,\qq)=\li.
\end{split}
\end{equation}
Thus, in any case $\frac{d^2}{dx^2}r(p,x)=0$ if $x$ belongs to a bridge.

We note that \cite[Lemma 3.6]{C7}, \cite[Equation (14)]{C7} and \cite[Proposition 3.9]{C7} are valid
for metrized graphs with possibly bridges.

If we consider \remref{rem notationRi} along with Equations (\ref{eqn bridge1}) and (\ref{eqn bridge2}),
the proof of \cite[Lemma 3.10]{C7} can be extended to the case $\ga$ with bridges.
\end{proof}
\lemref{lem term2} is crucial for our purposes.
\begin{lemma}\label{lem kirchhoff and term2}
For any metrized graph $\ga$, we have
\begin{equation*}\label{eqn kirchhoff and term2}
\begin{split}
4Kf(\ga)=v \cdot \sum_{e_i \in \,
\ee{\ga}}\frac{\li(R_{a_{i},p}-R_{b_{i},p})^2}{(\li+\ri)^2}
+ \sum_{e_i \in \,
\ee{\ga}}\frac{\ri}{\li + \ri} \sum_{p \in \vv{\ga}}\big(r(\pp,p)+r(\qq,p)\big).
\end{split}
\end{equation*}
\end{lemma}
\begin{proof}
First, we take summation of the second equality in \lemref{lem term2} over all vertices $p \in \vv{\ga}$:
\begin{equation*}\label{eqn kirchhoff and term2a}
\begin{split}
\sum_{p \in \vv{\ga}} \sum_{e_i \in \,
\ee{\ga}}\frac{\li(R_{a_{i},p}-R_{b_{i},p})^2}{(\li+\ri)^2}
= 2\sum_{p, \, q \in \vv{\ga}} r(p,q) -\sum_{e_i \in \,
\ee{\ga}}\frac{\ri}{\li + \ri} \sum_{p \in \vv{\ga}}\big(r(\pp,p)+r(\qq,p)\big).
\end{split}
\end{equation*}
Then the result follows from this equality, \lemref{lemrem2term} and the definition of $Kf(\ga)$.
\end{proof}

After having various technical lemmas,  we can state our first
main result. It describes the relation between the Kirchhoff index of $\ga$ and the Kirchhoff indexes of each
of $\oga_i$ that are obtained by contraction of $e_i \in \ee{\ga}$:
\begin{theorem}\label{thm main1}
Let $\ga$ be a metrized graph with at least $4$ vertices. Then we have
$$(v-4)Kf(\ga)=\sum_{e_i \in \ee{\ga}} \frac{\ri}{\li+\ri} Kf(\oga_i) - v \cdot y(\ga).$$
\end{theorem}
\begin{proof}
 We first subtract the equality given in \lemref{lem kirchhoff and term2} from the equality given in \lemref{lem Kirchhoff and res sum}.
 Then the proof follows from \eqnref{eqn expand ri} and \eqnref{eqn definition of x and y}.
\end{proof}
Next, we have another formula for Kirchhoff index.
\begin{proposition}\label{prop contraction for Kirchhoff}
For any metrized graph $\ga$ with $v$ vertices, we have
$$2Kf(\ga)= v \cdot y(\ga)+\sum_{e_i \in \ee{\ga}} \frac{\ri}{\li+\ri}\sum_{p \in \, \vv{\oga_i}} r_{\oga_i}(p,\opp).$$
\end{proposition}
\begin{proof}
The result is obtained by subtracting the formula in \thmref{thm main1} from \eqnref{eqn kirchhoff contraction1}.
\end{proof}
Note that \thmref{thm main1} is more advantageous to work with than \eqnref{eqn kirchhoff contraction1}, because we studied the term $y(\ga)$ previously \cite{C5} and showed that it has various properties.

Our goal for the rest of this section is to apply the contraction formula given in \thmref{thm main1} successively. To do this, we need the contraction formula for $y(\ga)$ for any metrized graph $\ga$ (see \thmref{thm contraction} below).  The contraction formula of $y(\ga)$ for bridgeless metrized graphs was shown in \cite[Theorem 4.12]{C5}. First, we need some preparatory work.

The following theorem was given in \cite[Theorem 4.8]{C6}. Note that we don't need the condition bridgeless as explained in the paragraph before the theorem in that paper (and as its proof shows). That is, we can give \cite[Theorem 4.8]{C6} with a minor correction in its statement as follows:
\begin{theorem}\label{thm resistance deletion and contraction ids3}
Let $\ga$ be a metrized graph. For any two vertices $p$ and $q$, we have
\begin{equation*}\label{eqn contraction general}
\begin{split}
(v-2) r(p,q)= \sum_{e_i \in \ee{\ga}} \frac{\ri}{\li+\ri} r_{\oga_i}(p,q).
\end{split}
\end{equation*}
\end{theorem}
Next, we apply \thmref{thm resistance deletion and contraction ids3} to the sum of effective resistances along with all edges.
Let $$r(\ga):=\sum_{e_i \in \ee{\ga}}\frac{\li \ri}{\li+\ri}.$$
Note that $r(\pp,\qq)=\frac{\li \ri}{\li+\ri}$ for any edge $e_i$ with end points $\pp$ and $\qq$.
\begin{theorem}\label{thm contraction for r}
Let $\ga$ be a metrized graph. Then, we have
\begin{equation*}\label{eqn contraction for r}
\begin{split}
(v-2) r(\ga)= \sum_{e_i \in \ee{\ga}} \frac{\ri}{\li+\ri} r(\oga_i).
\end{split}
\end{equation*}
\end{theorem}
\begin{proof}
Let $e_j$ be an edge with end points $p_j$ and $q_j$. Applying \thmref{thm resistance deletion and contraction ids3} to the vertices
$p_j$ and $q_j$ gives
\begin{equation*}\label{eqn contraction generalej}
\begin{split}
(v-2) r(p_j,q_j)= \sum_{e_i \in \ee{\ga}} \frac{\ri}{\li+\ri} r_{\oga_i}(p_j,q_j).
\end{split}
\end{equation*}
where $v$ is the number of vertices in $\ga$.
Now, if we take the summation of above equality over all edges $e_j$ in $\ga$ and use the definition of $r(\ga)$, we obtain
\begin{equation*}\label{eqn contraction generalej2}
\begin{split}
(v-2) r(\ga) &= \sum_{e_i \in \ee{\ga}} \frac{\ri}{\li+\ri} \sum_{e_j \in \ee{\ga}} r_{\oga_i}(p_j,q_j)\\
             &= \sum_{e_i \in \ee{\ga}} \frac{\ri}{\li+\ri} \sum_{e_j \in \ee{\oga_i}} r_{\oga_i}(p_j,q_j), \qquad \text{since $r_{\oga_i}(p_i,q_i)=0$}.\\
             &= \sum_{e_i \in \ee{\ga}} \frac{\ri}{\li+\ri} r(\oga_i).
\end{split}
\end{equation*}
This gives what we want to show.
\end{proof}
Note that \thmref{thm contraction for r} for bridgeless metrized graphs was given in \cite[Corollary 4.13]{C5}. But we show here that it holds for any metrized graphs possibly with bridges.

Similarly, the following theorem for bridgeless metrized graphs was given in \cite[Theorem 4.12]{C5}.
\begin{theorem}\label{thm contraction}
Let $\ga$ be a metrized graph with $v$ vertices. Then we have
\begin{align*}
(v-2) x(\ga) &=\sum_{e_i \in \ee{\ga}}\frac{\ri}{\li+\ri}x(\oga_i), & \text{  and   } \quad
(v-2) y(\ga) &=\sum_{e_i \in \ee{\ga}}\frac{\ri}{\li+\ri}y(\oga_i).
\end{align*}
\end{theorem}
\begin{proof}
Let $B(\ga)=\{ e_{i_1}, \, e_{i_2}, \dots, e_{i_t}  \}$ be the set of bridges in $\ga$.
Let $\beta$ be the metrized graph obtained from $\ga$ by contracting all bridges in $\ga$.

We first note that if $e_i$ is a bridge, using \remref{rem notationRi} we obtain
$\frac{\li^2\ri}{(\li+\ri)^2} \longrightarrow 0$, $\frac{\li \ri^2}{(\li+\ri)^2} \longrightarrow \li$
and $\frac{\li(R_{a_i,p}-R_{b_i,p})^2}{(\li+\ri)^2} \longrightarrow \li$. Therefore, considering the definition of $x(\ga)$ in \eqnref{eqn definition of x and y} we conclude that bridges in $\ga$ does not contribute to $x(\ga)$. Moreover, $R_j(\ga)=R_j(\beta)$ if $e_j$ is not a bridge. Hence,
\begin{equation}\label{eqn contraction xa}
\begin{split}
&x(\ga)=x(\oga_i) \quad \text{when $e_i$ is a bridge, and so  }    x(\ga)=x(\beta).\\
&x(\oga_i)=x(\overline{\beta}_i) \quad \text{when $e_i$ is not a bridge}.
\end{split}
\end{equation}
We use \eqnref{eqn contraction xa} and \remref{rem notationRi} in the second equality below:
\begin{equation}\label{eqn contraction xb}
\begin{split}
\sum_{e_i \in \ee{\ga}} \frac{\ri}{ \li + \ri } x(\oga_i) & = \sum_{e_i \in \ee{\ga}-B(\ga)} \frac{\ri}{ \li + \ri } x(\oga_i)+\sum_{e_i \in B(\ga)} \frac{\ri}{ \li + \ri } x(\oga_i) \\
& = \sum_{e_i \in \ee{\ga}-B(\ga)} \frac{\ri}{ \li + \ri } x(\beta)+\sum_{e_i \in B(\ga)} x(\beta), \\
&=(v-t-2) x(\beta)+t\cdot x(\beta), \qquad \text{using \cite[Theorem 4.12]{C5} for $\beta$.} \\
&=(v-2) x(\ga), \qquad \text{by using \eqnref{eqn contraction xa}}.
\end{split}
\end{equation}
This proves the first equality in the theorem. Next, we prove the second equality.
We first note that $r(\ga)=x(\ga)+y(\ga)$ for any metrized graph $\ga$.

On one hand, by \thmref{thm contraction for r} we have
\begin{equation}\label{eqn r x and y}
\begin{split}
(v-2)r(\ga)=\sum_{e_i \in \ee{\ga}} \frac{\ri}{ \li + \ri } r(\oga_i)=\sum_{e_i \in \ee{\ga}} \frac{\ri}{ \li + \ri } \big( x(\oga_i)+ y(\oga_i) \big).
\end{split}
\end{equation}
On the other hand, by the first equality that we just proved for $x(\ga)$
$$(v-2)r(\ga)=(v-2)x(\ga)+(v-2)y(\ga)=(v-2)y(\ga)+\sum_{e_i \in \ee{\ga}} \frac{\ri}{ \li + \ri }x(\oga_i).$$
Thus, the second equality in the theorem follows from this equality and \eqnref{eqn r x and y}.
\end{proof}
When the number of vertices is $2$ or $3$, we know the exact relation between $Kf(\ga)$ and $y(\ga)$.
\begin{corollary}\label{cor Kirchhoff v=2 or 3}
For any metrized graph $\ga$ with $v$ vertices. Then we have
$$ Kf(\ga)=y(\ga), \quad \text{if $v=2$}. \qquad Kf(\ga)=2y(\ga), \quad \text{if $v=3$}.$$
\end{corollary}
\begin{proof}
When $v=2$, $\oga_i$ has only one vertex. In this case, $Kf(\oga_i)=0$ for each edge $e_i$. Then \thmref{thm main1} gives
that $Kf(\ga)=y(\ga)$.

When $v=3$, $\oga_i$ has two vertices, so we have  $Kf(\oga_i)=y(\oga_i)$ by the first equality. Thus, \thmref{thm main1} gives
$$-Kf(\ga)=\sum_{e_i \in \ee{\ga}} \frac{\ri}{\li+\ri} Kf(\oga_i) - 3 \cdot y(\ga)=
\sum_{e_i \in \ee{\ga}} \frac{\ri}{\li+\ri} y(\oga_i) - 3 \cdot y(\ga)
= -2y(\ga), $$
where the last equality follows from \thmref{thm contraction}.
This completes the proof.
\end{proof}

For any integer $1 \leq k \leq v-2$, if an edge $e_{i_k}$ is not a self loop in $\oga_{i_1,i_2, \dots, i_{k-1}}$, then $\#(\vv{\oga_{i_1,i_2, \dots, i_k}})= \#(\vv{\oga_{i_1,i_2, \dots, i_{k-1}}})-1$. We call $\oga_{i_1,i_2, \dots, i_{k}}$ be an \textit{admissible contraction} of $\ga$, if it is obtained from $\ga$ by contracting edges with distinct end points at each step.
We have $\#(\vv{\oga_{i_1,i_2, \dots, i_{k}}})=v-k$ iff $\vv{\oga_{i_1,i_2, \dots, i_{k}}}$ is an admissible contraction of $\ga$. Note that we have $\frac{\ri}{\li+\ri}=0$ for a self loop, so contraction of self loops can be neglected in contraction identities. Therefore, we restrict ourselves to the admissible contractions only.

Now, we successively apply the contraction identity given in \thmref{thm contraction} as follows:
\begin{theorem}\label{thm succesive contraction of x and y}
Let $\ga$ be metrized graph with $v \geq 3$ vertices, and let $k$ be an integer with $1 \leq k \leq v-2$. For admissible contractions, we have
\begin{equation*}
\begin{split}
\frac{(v-2)!}{(v-k-2)!}x(\ga) &= \sum_{\substack{e_{i_1} \in
\\ \ee{\ga}}}\frac{R_{i_1}}{L_{i_1}+R_{i_1}} \sum_{\substack{e_{i_2} \in
\\ \ee{\oga_{i_1}}}}\frac{R_{i_2}}{L_{i_2}+R_{i_2}}
\; \dots
\sum_{ \substack{e_{i_k} \in
\\ \ee{\oga_{i_1, \dots, i_{k-1}}}} }
\frac{R_{i_k}}{L_{i_k}+R_{i_k}} x(\oga_{i_1,\dots, i_k}),\\
\frac{(v-2)!}{(v-k-2)!}y(\ga) &= \sum_{\substack{e_{i_1} \in
\\ \ee{\ga}}}\frac{R_{i_1}}{L_{i_1}+R_{i_1}} \sum_{\substack{e_{i_2} \in
\\ \ee{\oga_{i_1}}}}\frac{R_{i_2}}{L_{i_2}+R_{i_2}}
\; \dots
\sum_{ \substack{e_{i_k} \in
\\ \ee{\oga_{i_1, \dots, i_{k-1}}}} }
\frac{R_{i_k}}{L_{i_k}+R_{i_k}} y(\oga_{i_1,\dots, i_k}).
\end{split}
\end{equation*}
\end{theorem}
\begin{proof}
We have $\frac{R_{i_j}}{L_{i_j}+R_{i_j}}=0$ for an edge $e_{i_j}$ that is a self loop. Thus, contraction of self loops
does not contribute to sums in contraction identities.
Applying \thmref{thm contraction} inductively gives the result.
\end{proof}
Note that \thmref{thm succesive contraction of x and y} generalizes the similar results in \cite{C5} to any metrized graph.

Next, we take the advantage of the contraction formula to derive \thmref{thm main2} which is our second main result.
It describes how Kirchhoff index changes under successive edge contractions.
\begin{theorem}\label{thm main2}
Let $\ga$ be metrized graph with $v \geq 5$ vertices.
and let $k$ be an integer with $1 \leq k \leq v-4$. For admissible contractions, we have
\begin{equation*}
\begin{split}
Kf(\ga)&= \frac{(v-4-k)!}{(v-4)!}\sum_{\substack{e_{i_1} \in
\\ \ee{\ga}}}\frac{R_{i_1}}{L_{i_1}+R_{i_1}} \sum_{\substack{e_{i_2} \in
\\ \ee{\oga_{i_1}}}}\frac{R_{i_2}}{L_{i_2}+R_{i_2}}
\; \dots
\sum_{ \substack{e_{i_k} \in
\\ \ee{\oga_{i_1, \dots, i_{k-1}}}} }
\frac{R_{i_k}}{L_{i_k}+R_{i_k}} Kf(\oga_{i_1,\dots, i_k})\\
& \qquad
-\frac{\big(v^2-(k+2)v+k-1 \big)k}{(v-k-2)(v-k-3) }y(\ga).
\end{split}
\end{equation*}
In particular, if $k=v-4$, we have
\begin{equation*}
\begin{split}
Kf(\ga)&= \frac{1}{(v-4)!}\sum_{\substack{e_{i_1} \in
\\ \ee{\ga}}}\frac{R_{i_1}}{L_{i_1}+R_{i_1}} \sum_{\substack{e_{i_2} \in
\\ \ee{\oga_{i_1}}}}\frac{R_{i_2}}{L_{i_2}+R_{i_2}}
\; \dots
\sum_{ \substack{e_{i_{v-4}} \in
\\ \ee{\oga_{i_1, \dots, i_{v-5}}}} }
\frac{R_{i_{v-4}}}{L_{i_{v-4}}+R_{i_{v-4}}} Kf(\oga_{i_1,\dots, i_{v-4}})\\
& \qquad
-\frac{(3v-5)(v-4)}{ 2}y(\ga).
\end{split}
\end{equation*}
\end{theorem}
\begin{proof}
The proof follows by successive application of \thmref{thm main1} for each $Kf(\oga_{i_1,\dots, i_k})$ and \thmref{thm contraction} for each $y(\oga_{i_1,\dots, i_k})$. One should be careful about determining the coefficient of $y(\ga)$ after each contraction step. Note that we can compute the coefficient of $y(\ga)$ at the $k$-th contraction step with the help of the following identity:
$$\frac{v}{v-4}+\sum_{i=1}^{k-1}\frac{v-i}{v-4-i}\prod_{j=1}^{i}\frac{v-1-j}{v-3-j}=\frac{\big(v^2-(k+2)v+k-1 \big)k}{(v-k-2)(v-k-3) }.$$
\end{proof}
Note that $\oga_{i_1,\dots, i_{v-4}}$ has $4$ vertices. Therefore,  it is important to know the relation between $Kf(\ga)$ and $y(\ga)$ when $\ga$ has $4$ edges to derive further conclusions from \thmref{thm main2}. Although the exact relation as in \corref{cor Kirchhoff v=2 or 3} is not possible in general, we can have upper and lower bounds of $Kf(\ga)$ in terms of $y(\ga)$. This is what we show below.
First, we recall some facts.

Suppose the set of vertices for an admissible contraction $\oga_{i_1,i_2, \dots, i_{v-2}}$ of $\ga$ is $\{p', q'\}$.
Let $m$ vertices of $\ga$ are contracted into $p'$ and the remaining $k$ vertices are contracted into $q'$. Then both $m$ and $k$
are positive integers with $m+k=v$, where $v$ is the number of vertices in $\ga$.

Next, we state a corollary to \thmref{thm succesive contraction of x and y}. It generalizes the relevant result from \cite{C5} to any metrized graph possibly with bridges.
\begin{corollary}\label{cor contr for y k=v-2}
Let $\ga$ be metrized graph with $v \geq 3$ vertices. For admissible contractions $\oga_{i_1, \dots, i_{v-2}}$, we have
$$(v-2)! y(\ga) = \sum_{\substack{e_{i_1} \in
\\ \ee{\ga}}}\frac{R_{i_1}}{L_{i_1}+R_{i_1}} \sum_{\substack{e_{i_2} \in
\\ \ee{\oga_{i_1}}}}\frac{R_{i_2}}{L_{i_2}+R_{i_2}}
\; \dots
\sum_{ \substack{e_{i_{v-2}} \in
\\ \ee{\oga_{i_1, \dots, i_{v-3}}}} }
\frac{R_{i_{v-2}}}{L_{i_{v-2}}+R_{i_{v-2}}} r_{\oga_{i_1,\dots, i_{v-2}}}(p',q').$$
\end{corollary}
\begin{proof}
First we note that $y(\oga_{i_1, \dots, i_{v-2}})=r_{\oga_{i_1,\dots, i_{v-2}}}(p',q')$ by the proof of \cite[Proposition 5.8]{C5}.

Then the result follows from \thmref{thm succesive contraction of x and y} with $k=v-2$.
\end{proof}
We recall another contraction formula for the Kirchhoff index.
\begin{lemma}\cite[Lemma 5.2]{C6}\label{lem Kirchoff index successive contraction id1}
Let $\ga$ be a metrized graph with $v$ vertices, and let $m$ and $k$ be defined as above. For any admissible contraction $\oga_{i_1,i_2, \dots, i_{v-2}}$, we have
$$Kf(\ga) = \frac{1}{(v-2)!} \sum_{\substack{e_{i_1} \in \\ \ee{\ga}}}\frac{R_{i_1}}{L_{i_1}+R_{i_1}}
\dots \sum_{\substack{e_{i_{v-2}} \in \\ \ee{\oga_{i_1, \dots, i_{v-3}}}}}
\frac{R_{i_{v-2}}}{L_{i_{v-2}}+R_{i_{v-2}}}  m \cdot k \cdot r_{\oga_{i_1,\dots, i_{v-2}}}(p',q').$$
\end{lemma}
The following upper bound was given in \cite[Equation 21]{C6} for regular graphs that are bridgeless.
Now, we have it without any restriction on $\ga$:
\begin{corollary}\label{cor Kirchhoff upper bound}
For any metrized graph $\ga$ with $v$ vertices, we have
$$Kf(\ga) \leq \frac{v^2}{4} y(\ga).$$
\end{corollary}
\begin{proof}
When $m+k=v$ for any two positive integers $m$ and $k$, the maximum of $m \cdot k$ is at most $\frac{v^2}{4}$.
Then the proof follows from \lemref{lem Kirchoff index successive contraction id1} and \corref{cor contr for y k=v-2}.
\end{proof}
\begin{lemma}\label{lem Kirchhoff v=4}
Let $\ga$ be a metrized graph with $4$ vertices. Then we have
$$3y(\ga) \leq  Kf(\ga) \leq 4  y(\ga).$$
\end{lemma}
\begin{proof}
We apply the contraction formula given in \lemref{lem Kirchoff index successive contraction id1} to $\ga$.
Since $m+k=4$ and both $m$ and $k$ are positive integers, we either have $m \cdot k =3$ or $m \cdot k =4$.
Thus, the inequalities in the lemma follows from \corref{cor contr for y k=v-2}.
\end{proof}
Now, using \lemref{lem Kirchhoff v=4} for $\oga_{i_1,\dots, i_{v-4}}$, \thmref{thm main2} and \thmref{thm succesive contraction of x and y} with $k=v-4$, we derive the following proposition:
\begin{proposition}\label{prop Kirchhooff uplow}
For any metrized graph with $v \geq 4$, we have
$$ (v-1)y(\ga) \leq  Kf(\ga) \leq \frac{v^2-3v+4}{2}  y(\ga).$$
\end{proposition}
We note that when $v \geq 5$ \corref{cor Kirchhoff upper bound} gives better upper bounds then \propref{prop Kirchhooff uplow}.

Next, we give an example to illustrate how the contraction formula in \thmref{thm main2} can be used.
\begin{figure}
\centering
\includegraphics[scale=0.5]{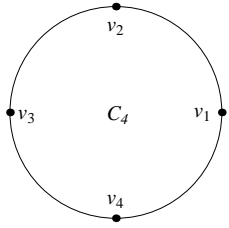} \caption{Circle graph with $4$ vertices} \label{fig circle}
\end{figure}

\textbf{Example I:} Let $C_v$ be the circle graph with $v$ vertices and $v$ edges. \figref{fig circle} illustrates $C_4$.
Suppose each edge length of the metrized graph $\ga=C_v$ is equal to $1$. Then $\ell(C_v)=v$, and we have $Kf(C_4)=5$ by direct computation. Moreover,
$\tg=\frac{1}{12}\ell(C_v)$ by \cite[Corollary 2.17]{C4}, $\tg=\frac{1}{12}\ell(C_v)-\frac{x(\ga)-y(\ga)}{6}$ by \cite[Equation 20]{C5},
$x(\ga)+y(\ga)=\frac{v-1}{v}\ell(C_v)$ by \cite[Lemma 6.3]{C6}. Thus, $x(\ga)=y(\ga)=\frac{v-1}{2}$.

Since $\oga_{i_1,\dots, i_{v-4}}=C_4$ for every admissible contraction of $\ga$ in this case, applying \thmref{thm main2} with $k=v-4$ gives
$Kf(C_v)=\frac{v(v^2-1)}{12}$. This agrees with the result obtained in \cite[Equation (5)]{LNT}.

\section{Trees, When Kirchhoff Index is Wiener Index}\label{sec tree}

In this section, we restrict ourselves to tree metrized graphs. A tree graph is a connected graph with no cycle. That is, each edge in a tree graph is a bridge. We rewrite many of the results from \secref{sec contraction} for the tree metrized graphs. In this way, we obtain new formulas for the Wiener index of tree graphs, and give new proofs to some previously know formulas for Wiener index.

Let $d(p,q)$ denote the distance between the vertices $p$ and $q$ in $\vv{\ga}$. Then the Wiener index of $\ga$
is defined as follows (see \cite[page 211]{DEG} and the references therein):
$$W(\ga):=\frac{1}{2} \sum_{p, \, q \in \vv{\ga}}d(p,q).$$
When $\ga$ is a tree, $d(p,q)=r(p,q)$ for each vertices $p$ and $q$, where $r(x,y)$ is the resistance function on $\ga$.
Therefore,
\begin{equation}\label{eqn wiener and Kirchhoff}
\begin{split}
W(\ga)=Kf(\ga) \quad \text{if  $\ga$ is a tree}.
\end{split}
\end{equation}

When $\ga$ is a tree, \lemref{lem term2} can be restated as follows
\begin{lemma}\label{lem term2fortrees}
Let $\ga$ be a metrized graph that is a tree with $v$ vertices. For any $p
\in \vv{\ga}$, we have
$$\elg = \sum_{q \in \vv{\ga}}(2-\va(q))r(p,q).$$
In particular, if each edge length is equal to $1$, we have
$$v-1 = \sum_{q \in \vv{\ga}}(2-\va(q))r(p,q).$$
\end{lemma}
\begin{proof}
Each edge is a bridge in $\ga$ as it is a tree.
Thus, we have $(R_{a_{i},p}-R_{b_{i},p})^2=\ri^2$ for each edge in $\ga$, and so $\frac{\li(R_{a_{i},p}-R_{b_{i},p})^2}{(\li+\ri)^2}=\frac{\li \ri^2}{(\li+\ri)^2}$.
We have $\frac{\li \ri^2}{(\li+\ri)^2} \longrightarrow \li$ and $\frac{\li}{\li+\ri}\longrightarrow 0$ as $\ri \longrightarrow \infty$.
Therefore, the first equality in the lemma follows from the first equality given in \lemref{lem term2}.
When $\li=1$ for every $e_i \in \ee{\ga}$, the second equality in the lemma is obtained by using the fact that $\elg = e =v-1$, where $e$ is the number of edges of $\ga$.
\end{proof}

\begin{theorem}\label{thm wiener tree1}
Let $\ga$ be a metrized graph that is a tree with $v$ vertices. Then we have
$$W(\ga)=\frac{1}{4} \Big[ v \cdot \elg +  \sum_{p, \, \, q \in \vv{\ga}} \va(q)r(p,q) \Big] .$$
In particular, if each edge length is equal to $1$, we have
$$W(\ga)=\frac{1}{4} \Big[ v(v-1)+  \sum_{p, \, \, q \in \vv{\ga}} \va(q)r(p,q) \Big] .$$
\end{theorem}
\begin{proof}
We take the summation of the equalities given in \lemref{lem term2fortrees} over all vertices $p \in \vv{\ga}$. Then we obtain the result
by using the definition of $W(\ga)$.
\end{proof}
Note that the result given in \thmref{thm wiener tree1} was known in the literature for trees with equal edge lengths
(see \cite[page 217]{DEG} and the references therein).

Now, we can state our first main result for trees:
\begin{theorem}\label{thm wiener tree2}
Let $\ga$ be a metrized graph that is a tree with $v$ vertices. Then we have
$$W(\ga)=\frac{2v-1}{4}\elg+ \frac{1}{8} \sum_{p, \, \, q \in \vv{\ga}} \va(p) \va(q) r(p,q).$$
In particular, if each edge length is equal to $1$, we have
$$W(\ga)=\frac{(2v-1)(v-1)}{4}+ \frac{1}{8} \sum_{p, \, \, q \in \vv{\ga}} \va(p) \va(q) r(p,q).$$
\end{theorem}
\begin{proof}
We first multiply both sides of the first equality given in \lemref{lem term2fortrees} by $2-\va(p)$. Then
we take the summation of both sides over all vertices $p \in \vv{\ga}$. This gives
$$\elg \sum_{p \in \vv{\ga}} (2- \va(p) ) = \sum_{p, \, \, q \in \vv{\ga}} (2-\va(p)) (2-\va(q)) r(p,q).$$
Since $\sum_{p \in \vv{\ga}} (2- \va(p) )=2v-2e=2$, we have
\begin{equation*}\label{eqn wiener2}
\begin{split}
2 \elg &= \sum_{p, \, \, q \in \vv{\ga}} (2-\va(p)) (2-\va(q)) r(p,q). \quad \text{Using the definition of $W(\ga)$ gives} \\
&= 8W(\ga)+\sum_{p, \, \, q \in \vv{\ga}} \va(p) \va(q) r(p,q)-4\sum_{p, \, \, q \in \vv{\ga}} \va(q) r(p,q), \\
&=-8W(\ga)+4v\cdot \elg+\sum_{p, \, \, q \in \vv{\ga}} \va(p) \va(q) r(p,q), \quad \text{by  \thmref{thm wiener tree1}.}
\end{split}
\end{equation*}
This gives the first equality. The second equality follows from the first one by using the fact that $\elg=v-1$ when each edge length is equal to $1$.
\end{proof}
A discussion similar to the proof of \lemref{lem term2fortrees} gives
\begin{equation}\label{eqn x and y for tree}
\begin{split}
x(\ga)=0  \, \text{   and   }  \, y(\ga)=\elg  \, \text{if $\ga$ is a tree.}
\end{split}
\end{equation}

Next, we restate \thmref{thm main2} for a tree:
\begin{theorem}\label{thm main2 for tree}
Let metrized graph $\ga$ be a tree with $v \geq 5$ vertices.
and let $k$ be an integer with $1 \leq k \leq v-4$. For admissible contractions, we have
\begin{equation*}
\begin{split}
W(\ga)&= \frac{(v-4-k)!}{(v-4)!}\sum_{\substack{e_{i_1} \in
\\ \ee{\ga}}} \sum_{\substack{e_{i_2} \in
\\ \ee{\oga_{i_1}}}}
\; \dots
\sum_{ \substack{e_{i_k} \in
\\ \ee{\oga_{i_1, \dots, i_{k-1}}}} }
 W(\oga_{i_1,\dots, i_k})
-\frac{\big(v^2-(k+2)v+k-1 \big)k}{(v-k-2)(v-k-3) } \elg.
\end{split}
\end{equation*}
In particular, if $k=v-4$, we have
\begin{equation*}
\begin{split}
W(\ga)&= \frac{1}{(v-4)!}\sum_{\substack{e_{i_1} \in
\\ \ee{\ga}}} \sum_{\substack{e_{i_2} \in
\\ \ee{\oga_{i_1}}}}
\; \dots
\sum_{ \substack{e_{i_{v-4}} \in
\\ \ee{\oga_{i_1, \dots, i_{v-5}}}} }
 W(\oga_{i_1,\dots, i_{v-4}})
-\frac{(3v-5)(v-4)}{ 2} \elg.
\end{split}
\end{equation*}
\end{theorem}
\begin{proof}
Since each edge is a bridge, we have $\frac{\ri}{\li+\ri} \longrightarrow 1$ for each edge $e_i$ in $\ga$.
Then the result follows from \thmref{thm main2}, \eqnref{eqn x and y for tree} and \eqnref{eqn wiener and Kirchhoff}.
\end{proof}
\begin{figure}
\centering
\includegraphics[scale=0.8]{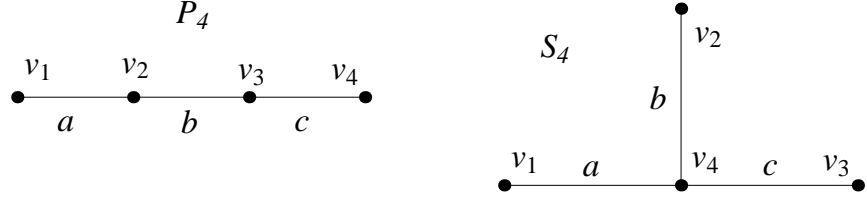} \caption{Path and star graphs with $4$ vertices} \label{fig pathstar}
\end{figure}
To derive further results about $W(\ga)$ by using \thmref{thm main2 for tree}, we need to understand the Wiener index of $\oga_{i_1,\dots, i_{v-4}}$ which is a tree with $4$ vertices. Thus, we consider \lemref{lem wiener for v=4} below.

Let $S_n$ and $P_n$ be star and path metrized graphs on $n$ vertices, respectively. \figref{fig pathstar} illustrates $S_4$ and $P_4$.
\begin{lemma}\label{lem wiener for v=4}
Suppose metrized graph $\ga$ is a tree with $4$ vertices. Then $\ga$ is either $S_4$ or $P_4$.
Moreover, $W(P_4)=3(a+b+c)+b$ and $W(S_4)=3(a+b+c)$, where edge lengths are as in \figref{fig pathstar}.
\end{lemma}
\begin{proof}
A direct computation gives the result.
\end{proof}

Now, we can state our second main result for trees:
\begin{theorem}\label{thm main3 for tree}
Let metrized graph $\ga$ be a tree with $v$ vertices. Suppose each edge length of $\ga$ is $1$. Then we have
\begin{equation*}
\begin{split}
W(\ga) &= (v-1)^2+\sum_{\substack{\{ e_{i_1}, \, e_{i_2}, \, \cdots, \,  e_{i_{v-4}} \} \subset \ee{\ga} \\
\oga_{i_1,\dots, i_{v-4}}=P_4}} 1,\\
&= (v-1)^2+\sum_{\substack{\{ e_{i_1}, \, e_{i_2}, \, e_{i_3} \} \subset \ee{\ga} \\
e_{i_1}, e_{i_2}, e_{i_3} \in P \\ \text{P is a path in $\ga$}}} 1.
\end{split}
\end{equation*}
The last summation is taken over all subsets $\{ e_{i_1}, \, e_{i_2}, \, e_{i_3} \}$ of $\ee{\ga}$ such that the edges $e_{i_1}$, $e_{i_2}$ and $e_{i_3}$ are parts of a path in $\ga$.
\end{theorem}
\begin{proof}
Applying \lemref{lem wiener for v=4} for this case, we obtain $W(\oga_{i_1,\dots, i_{v-4}})=10$ if $\oga_{i_1,\dots, i_{v-4}}=P_4$, and $W(\oga_{i_1,\dots, i_{v-4}})=9$ if $\oga_{i_1,\dots, i_{v-4}}=S_4$.

We note that
\begin{equation*}
\begin{split}
\frac{1}{(v-4)!} \sum_{\substack{e_{i_1} \in
\\ \ee{\ga}}} \sum_{\substack{e_{i_2} \in
\\ \ee{\oga_{i_1}}}}
\; \dots
\sum_{ \substack{e_{i_{v-4}} \in
\\ \ee{\oga_{i_1, \dots, i_{v-5}}}} }
1 = \frac{(v-1)(v-2)(v-3)}{6}.
\end{split}
\end{equation*}
Then we use \thmref{thm main2 for tree} with $\elg=(v-1)$ to obtain
\begin{equation*}
\begin{split}
W(\ga) = (v-1)^2+\frac{1}{(v-4)!}\sum_{\substack{e_{i_1}, \, e_{i_2}, \, \cdots, \,  e_{i_{v-4}} \in \ee{\ga} \\
\oga_{i_1,\dots, i_{v-4}}=P_4}} 1 = (v-1)^2+\sum_{\substack{\{ e_{i_1}, \, e_{i_2}, \, \cdots, \,  e_{i_{v-4}} \} \subset \ee{\ga} \\
\oga_{i_1,\dots, i_{v-4}}=P_4}} 1 .
\end{split}
\end{equation*}
We have the second equality above, because the number of permutations of $v-4$ edges $e_{i_1}, \, e_{i_2}, \, \cdots, \,  e_{i_{v-4}}$ is $(v-4)!$.
This gives the first equality in the theorem.

The second equality in the theorem follows from the first one.
\end{proof}
Note that \thmref{thm main3 for tree} in a sense gives information about how far a graph is away from being a star graph.

As a corollary to \thmref{thm main3 for tree}, we obtain the following well-known result:
\begin{corollary}\label{cor wiener upper and lower}
For any metrized graph $\ga$ with $v$ vertices, we have
$$(v-1)^2 = W(S_v) \leq W(\ga) \leq W(P_v)=\frac{v(v^2-1)}{6}.$$
\end{corollary}
\begin{proof}
No path in $S_v$ can contain $3$ edges, so $W(S_v)=(v-1)^2$ by using \thmref{thm main3 for tree}. Since this is the case with minimum value $0$ of the summation in the formula of \thmref{thm main3 for tree}, we obtain $W(S_v) \leq W(\ga)$.

On the other hand, any $3$ edges in $P_v$ is part of a path in $P_v$, namely the path is $P_v$ itself. Thus, the summation in the formula of \thmref{thm main3 for tree} is $\binom{v-1}{3}$, and so $W(P_v)=(v-1)^2+\binom{v-1}{3}$ by \thmref{thm main3 for tree}. We note that $\binom{v-1}{3}$ is the maximum value of the summation, so $W(\ga) \leq W(P_v)$.
\end{proof}

We recall the following result due to Doyle and Graver \cite{DG} to compare with \thmref{thm main3 for tree}:
\begin{theorem}\cite[Theorem 9]{DEG} \label{thm away from path}
Let metrized graph $\ga$ be a tree with $v$ vertices, and let $v_1$, $v_2$, $\dots$, $v_{\va(p)}$ be the number vertices in the connected components of the graph obtained from $\ga$ by deleting the edges connected to a vertex $p$. Then
$$W(\ga)=\frac{v(v^2-1)}{6}- \sum_{\substack{ p \in \vv{\ga}, \\  \va(p) \geq 3 }} \, \, \, \sum_{1 \leq i < j < k \leq \va(p)} v_i v_j v_k.$$
\end{theorem}
Note that \thmref{thm away from path} somewhat explains how far a graph is away from being a path graph.

Next, we restate \lemref{lem Kirchoff index successive contraction id1} for trees. For an edge $e_i$ with end points $\pp$ and $\qq$, let
$m_i$ be the number of vertices that are in the connected component of $\ga-e_i$ containing $\pp$, and let
$k_i$ be the number of vertices that are in the connected component of $\ga-e_i$ containing $\qq$. Then we have $m_i+k_i=v$.
\begin{theorem}\label{thm wiener for tree1 multip}
Let metrized graph $\ga$ be a tree with $v$ vertices. Suppose each edge length of $\ga$ is $1$. Then we have
\begin{equation*}
\begin{split}
W(\ga)= \sum_{e_i \in \ee{\ga}} m_i \cdot k_i.
\end{split}
\end{equation*}
\end{theorem}
\begin{proof}
Each edge $e_i$ is bridge, so $\frac{\ri}{\li+\ri} \longrightarrow 0$. Moreover,
$\oga_{i_1,\dots, i_{v-2}}$ is the edge that is not contracted, so we have $r_{\oga_{i_1,\dots, i_{v-2}}}(p',q')=1$.
Therefore, \lemref{lem Kirchoff index successive contraction id1} implies
$$W(\ga)=\frac{1}{(v-2)!}\sum_{\substack{e_{i_1}, \, e_{i_2}, \, \cdots, \,  e_{i_{v-2}} \in \ee{\ga}}} m \cdot k,$$
where $m$ and $k$ are as defined before. Considering the permutations of the contracted edges, we can rewrite this as
\begin{equation*}
\begin{split}
W(\ga)=\sum_{\substack{ \{ e_{i_1}, \, e_{i_2}, \, \cdots, \,  e_{i_{v-2}} \}  \subset \ee{\ga}}} m \cdot k
= \sum_{e_i \in \ee{\ga}} m \cdot k.
\end{split}
\end{equation*}
Now, suppose $\oga_{i_1,\dots, i_{v-2}}=e_i$. Then the number of vertices contracted into $\pp$ is nothing but $m_i$, i.e., $m=m_i$. Similarly, the number of vertices contracted into $q'$ is $k_i$. That is, $k=k_i$.
This completes the proof.
\end{proof}
Note that \thmref{thm wiener for tree1 multip} was known previously \cite[page 218]{DEG} and \cite{HW}.

\textbf{Example II:}
Let $\beta_1$ and $\beta_2$  be the metrized graphs with $s+t+2$ and $s+t+3$ vertices, respectively. These are illustrated in \figref{fig ex1}. Suppose $s \geq 0$, $t \geq 0$ and each edge length in $\beta_1$ and $\beta_2$ is equal to 1. By applying \thmref{thm main3 for tree}, we obtain

$W(\beta_1)=(s+t+1)^2+\binom{s}{1} \binom{t}{1}=(s+t+1)^2+s\cdot t$.

$W(\beta_2)=(s+t+2)^2+\binom{s}{1} \binom{2}{1} \binom{t}{1}+\binom{s}{1} \binom{2}{2} \binom{t}{0}+\binom{s}{0} \binom{2}{2} \binom{t}{1}=(s+t+2)^2+2s\cdot t+s+t$.
\begin{figure}
\centering
\includegraphics[scale=0.6]{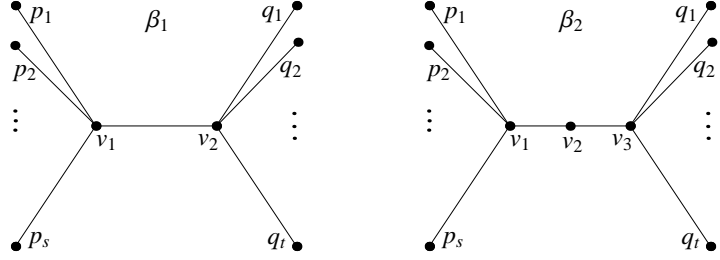} \caption{Tree metrized graphs $\beta_1$ and $\beta_2$.} \label{fig ex1}
\end{figure}

We note that these results agree with the results given in \cite[page 234]{DEG} (as $\beta_1=D(s+t+2,s,t)$ and $\beta_2=D(s+t+3,s,t)$, where $D(v,s,t)$ is the graph defined as in \cite[page 234]{DEG}).

\textbf{Example III:}
In this example, we work with metrized graphs $\beta_3$ and $\beta_4$ illustrated in  \figref{fig ex2}. $\beta_3$ and $\beta_4$ have $s+t+k+4$ and $s+t+k+m+4$ vertices, respectively. Suppose $s \geq 0$, $t \geq 0$, $k \geq 0$, $m \geq 0$ and each edge length in these graphs is equal to 1. By applying \thmref{thm main3 for tree}, we obtain
\begin{equation*}
\begin{split}
W(\beta_3)&=(s+t+k+3)^2+\binom{s}{1} \binom{2}{1} \binom{k}{1}+\binom{s}{1} \binom{2}{2} + \binom{2}{2} \binom{k}{1}
+\binom{s}{1} \binom{2}{1} \binom{t}{1}\\
&\quad +\binom{s}{1} \binom{2}{2} +  \binom{2}{2} \binom{t}{1} +\binom{k}{1} \binom{2}{1} \binom{t}{1}+\binom{k}{1} \binom{2}{2} + \binom{2}{2} \binom{t}{1}\\
&=(s+t+k+3)^2+2(sk+st+kt+s+t+k).
\end{split}
\end{equation*}
Now, to compute $W(\beta_4)$ we can use the computation used in obtaining $W(\beta_3)$. Namely, when we compute the number of three edges that are part of a path in $\beta_4$,  we can divide the edges in two groups: The ones having an end point in $\{ u_1, u_2, \cdots, u_m \}$
and the ones with no end points in this set.
\begin{equation*}
\begin{split}
W(\beta_4)&=(s+t+k+m+3)^2+2(sk+st+kt+s+t+k)+\binom{m}{1} \Big[ \binom{s}{1} + \binom{k}{1}+ \binom{t}{1} \Big]\\
&=(s+t+k+m+3)^2+2(sk+st+kt)+(m+2)(s+t+k).
\end{split}
\end{equation*}

\begin{figure}
\centering
\includegraphics[scale=0.6]{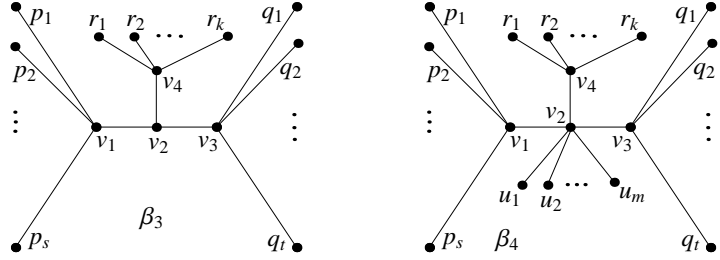} \caption{Tree metrized graphs $\beta_3$ and $\beta_4$.} \label{fig ex2}
\end{figure}

\textbf{Example IV:}
In this case, we work with metrized graph $\beta_5$ illustrated in  \figref{fig ex3}. $\beta_5$ has $v=s+t+k+m+n+5$ vertices. Suppose $s \geq 0$, $t \geq 0$, $k \geq 0$, $m \geq 0$, $n \geq 0$ and each edge length of $\beta_5$ is equal to 1. By applying \thmref{thm main3 for tree} and using the computation of $W(\beta_4)$, we obtain

$$W(\beta_5)=(v-1)^2+2(sk+st+kt+mn+kn+sn)+(n+2)(s+k)+(m+2)(s+k+t+1)+n(t+5).$$
The details are left as an exercise to the reader.
\begin{figure}
\centering
\includegraphics[scale=0.85]{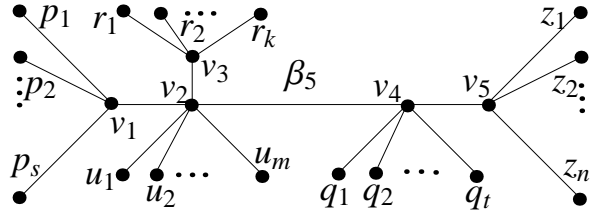} \caption{Tree metrized graph $\beta_5$.} \label{fig ex3}
\end{figure}

\textbf{Example V:}
In this case, we work with metrized graph $\beta_6$ illustrated in  \figref{fig ex3}. $\beta_6$ has $v=s+t+k+m+n+h+6$ vertices. Suppose $s \geq 0$, $t \geq 0$, $k \geq 0$, $m \geq 0$, $n \geq 0$, $h \geq 0$ and each edge length of $\beta_6$ is equal to 1. By applying \thmref{thm main3 for tree} and using the computation of $W(\beta_5)$, we obtain
\begin{equation*}
\begin{split}
W(\beta_6)&=(v-1)^2+2(sk+st+kt+mn+kn+sn)+(n+4)(s+k)+n(t+6)\\
&\quad +(m+2)(s+k+t+2)+h(3s+3k+2n+2m+t+6).
\end{split}
\end{equation*}
The details are left as an exercise to the reader.

\begin{figure}
\centering
\includegraphics[scale=0.85]{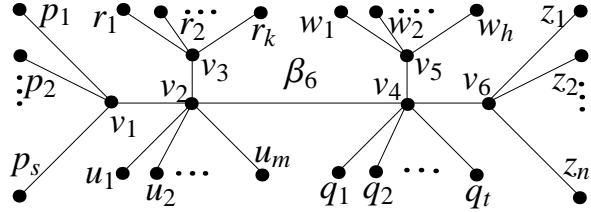} \caption{Tree metrized graph $\beta_6$.} \label{fig ex4}
\end{figure}

\textbf{Problem I:} Show that
the function $F:\NN^5 \longrightarrow \NN$ given by
$F(s,t,k,m,n)=(s+t+k+m+n+4)^2+2(sk+st+kt+mn+kn+sn)+(n+2)(s+k)+(m+2)(s+k+t+1)+n(t+5)$
takes every integer bigger than $557$, and that the only integers not assumed by $F$ are the following $89$ numbers:

$\{1,2,3,4,5,6,7,8,9,10,11,12,13,14,15,16,17,19,20,21,22,23,24,25,26,27,30,33,34,35,$
$36,37,38,39,41,43,45,47,49,51,52,53,55,56,60,61,69,73,75,77,78,79,81,83,85,87,89,91,$
$99,101,106,113,125,129,131,133,135,141,143,147,149,157,159,165,197,199,203,213,217,$
$219,281,285,293,301,325,357,501,509,557\}$.

We checked by a computer program that any integer not in the list above and less than $20000$ can be a value of $F$.

\textbf{Problem II:} Show that
the function $G:\NN^6 \longrightarrow \NN$ given by
$G(s,t,k,m,n,h)=(s+t+k+m+n+h+5)^2+2(sk+st+kt+mn+kn+sn)+(n+4)(s+k)+(m+2)(s+k+t+2)+n(t+6)+h(3s+3k+2n+2m+t+6) $
takes every integer bigger than $301$, and that the only integers not assumed by $G$ are the following $104$ numbers:

\footnotesize
$\{1,2,3,4,5,6,7,8,9,10,11,12,13,14,15,16,17,18,19,20,21,22,23,24,25,26,27,28,30,31, 32,33,34,35, \quad$ $36,37,38,39,40,41,43,44,45,47,48,49,50,51,52,53,54,55,56,59,60,61,64,66,69,70,71,72,73,75,77,78, \quad$ $79,81,83,85,87,89,91,95,98,99,101,102,106,113,119,124,127,129,131,133,135,139,141,143,147,149,157, \qquad $ $159,165,197,199,203,213,217,219,279,293,301\}$.
\normalsize

Again, we tested by a computer program that any integer not in the list above and less than $20000$ can be a value of $G$.

The following theorem was conjectured in \cite{LG} and \cite{GY}, and proved in both \cite{WY} and \cite{W}.
\begin{theorem}\label{thm wiener conj}
Except for exactly the following $49$ positive integers, every positive integer is the Wiener
index of some tree.

$\{2, 3, 5, 6, 7, 8, 11, 12, 13, 14, 15, 17, 19, 21, 22, 23, 24, 26, 27, 30, 33, 34, 37, 38, 39, 41,
43, 45, 47,$
$ 51, 53, 55, 60, 61, 69, 73, 77, 78, 83, 85, 87, 89, 91, 99, 101, 106, 113, 147,
159 \}$.
\end{theorem}
%

Note that a positive solution to any of Problem I and Problem II above will be another proof of \thmref{thm wiener conj}.



\textbf{Acknowledgements:} This work is supported by The Scientific and Technological Research Council of Turkey-TUBITAK (Project No: 110T686).

\end{document}